\documentclass[11pt,bezier]{article}
\usepackage{amsmath,amssymb,amsfonts,euscript,graphicx}

\newtheorem{preproof}{{\bf \indent Proof.}}

\newenvironment{proof}[1]{\begin{preproof}{\rm
               #1}\hfill{$\Box$}}{\end{preproof}}

\newtheorem{cor}{\bf\indent Corollary}[section]
\newtheorem{example}{\bf\indent Example}[section]
\newtheorem{thm}{{\bf\indent Theorem}}[section]

\newtheorem{remark}{\sc\indent Remark}[section]
\newtheorem{lem}{\bf\indent Lemma}[section]

\title{\bf \large Coloring in essential annihilating-ideal graphs\\ of commutative rings}

\author{{\normalsize   {\sc R. Nikandish${}^{\mathsf{a}}$\thanks{Corresponding author}, {\sc M. Mehrara${}^{\mathsf{b}}$ and {\sc M. J. Nikmehr${}^{\mathsf{b}}$}}  }
}\vspace{3mm}\\
{\footnotesize{${}^{\mathsf{a}}$\it Department of Mathematics, Jundi-Shapur University of Technology,}}\\
{\footnotesize{\rm P.O. BOX \rm{64615-334},
Dezful, Iran}}\\
{\footnotesize{${}^{\mathsf{b}}$\it Faculty of Mathematics, K.N. Toosi
University of Technology, }}\\
{\footnotesize{\rm P.O. BOX \rm{16315-1618}, Tehran, Iran}}\\
{\footnotesize{ $\mathsf{r.nikandish@ipm.ir}$}}\quad\quad
{\footnotesize{$\mathsf{nikmehr@kntu.ac.ir}$}}\quad\quad
{\footnotesize{ $\mathsf{mhrmahnoush@gmail.com}$}}\\
{\footnotesize{$\mathsf{}$ }}}
\date{}

\begin{document}

\maketitle


\begin{abstract}
{\small The essential annihilating-ideal graph $\mathcal{EG}(R)$ of a commutative unital ring $R$ is a simple graph whose vertices are non-zero ideals of $R$ with non-zero annihilator and there exists an edge between two distinct vertices $I,J$ if and only if $Ann(IJ)$ has a non-zero intersection with any non-zero ideal of $R$. In this paper, we show that $\mathcal{EG}(R)$ is weakly perfect, if $R$ is Noetherian and an explicit formula for the clique number of $\mathcal{EG}(R)$ is given.  Moreover, the structures of  all rings whose essential annihilating-ideal graphs have chromatic number $2$ are fully determined. Among other results, twin-free clique number and edge chromatic number of $\mathcal{EG}(R)$ are examined.
}
\end{abstract}
\begin{center}\section{Introduction}\end{center}

 Computing the vertex and edge chromatic number in a graph are examples of NP-complete problems in discrete structures which have many applications not only in real life but also in many branches of computer science. Although graph coloring is an old topic, it is still one of the most active areas in graph theory; for the most recent study in graph coloring see for instance \cite{cib}, \cite{jac}, \cite{nor},  \cite{st} and \cite{tardif}. In addition to wide range of applications, the complexity of computations  has caused considerable interest in characterizing these invariants for graphs  associated with algebraic structures, some
examples in this direction may be found in  \cite{bak}, \cite{dal}, \cite{ma} and \cite{nikan}. This paper is in this field and aims to investigate the  the vertex and edge coloring  in  essential annihilating-ideal graphs of commutative rings.

Throughout this paper, all rings $R$ are  commutative with identity. The sets of all maximal ideals, minimal prime ideals,  ideals with non-zero annihilator and  nilradical  of a ring $R$ are denoted by ${\rm Max}(R)$, ${\rm Min}(R)$,  $A(R)$ and  ${\rm Nil}(R)$, respectively. Also, by $Z(R)$, we mean the set of zero-divisors in $R$. If $I$ is an ideal of $R$, then we write $I\leq R$.
  If $B$ is a subset of $R$, then by $B^*$, we mean $B\setminus\{0\}$. A ring $R$ is said to be  \textit{reduced} if $0_R$ is the only nilpotent element of $R$.   A non-zero ideal $I$ of $R$ is called \textit{essential}, denoted by $I\leq_e R$, if $I$ has a non-zero intersection with any non-zero ideal of $R$. The \textit{socle} of
an $R$-module $M$, denoted by $Soc(M)$, is the sum of all simple submodules of $M$. If
there are no simple submodules, this sum is defined to be zero. It is well-known $\mathrm{Soc}(M)$ is the intersection of all essential submodules (see \cite[21.1]{wis}). For undefined notation or terminology from ring theory we refer the reader to \cite{ati}.

Let $G=(V,E)$ be a graph, where $V=V(G)$ is the set of vertices and $E=E(G)$ is the set of edges.   For a vertex $x \in V(G)$, the degree, open and closed neighborhood of $x$ are denoted by $deg(x)$, $N(x)$ and $N[x]$, respectively. The maximum degree of vertices of $G$ is denoted by $\Delta(G)$. The graph $H=(V_0,E_0)$ is a \textit{subgraph of} $G$ if $V_0\subseteq V$ and $E_0 \subseteq E$. Moreover, $H$ is called an \textit{induced subgraph by} $V_0$,  denoted by $G[V_0]$, if $V_0\subseteq V$ and $E_0=\{\{u,v\}\in E\, |\,u,v\in V_0\}$.   For two vertices $u$ and $v$ in $G$, the notation $u-v$ means that $u$ and $v$ are adjacent.  A complete graph of order $n$ and a complete bipartite graph with part sizes $m$ and $n$ are denoted by $K_n$, $K_{m,n}$, respectively. If the size of one of the parts is $1$ in a complete bipartite graph, then the graph is said to
be a \textit{star graph}. In a graph $G$, a set $S\subseteq V(G)$ is an \textit{independent set} if the subgraph induced by $S$ is totally disconnected.   Let $G_1$ and $G_2$ be two disjoint graphs. The \textit{join} of $G_1$ and $G_2$, denoted by $G_1\vee G_2$, is a graph with the vertex set $V(G_1\vee G_2)=V(G_1)\cup V(G_2)$ and edge set $E(G_1\vee G_2)=E(G_1)\cup E(G_2)\cup \{uv\,|\, u\in V(G_1), v\in V(G_2) \}$. A \textit{clique} of
$G$ is a maximal complete subgraph of $G$ and the number of
vertices in the largest clique of $G$, denoted by $\omega(G)$, is
called the \textit{clique number} of $G$. For  a graph $G$, let
$\chi(G)$ denote the \textit{vertex chromatic number} of $G$, i.e., the
minimal number of colors which can be assigned to the vertices of
$G$ in such a way that every two adjacent vertices have different
colors. Clearly, for every graph $G$, $\omega(G)\leq \chi(G)$. A graph $G$ is said to be \textit{weakly perfect} if $\omega(G)=\chi(G)$. Two distinct vertices $x,y$ are called \textit{true twins} if   $N[x]=N[y]$. The set $X\subseteq V(G)$ is called a \textit{twin-free clique in} $G$ if the subgraph induced by $X$
is a clique and for every $u,v \in X$ it follows that $N[u]\neq N[v]$, that is, the subgraph induced by $X$ is a clique and it contains no true twins. \textit{The twin-free clique number of} $G$, denoted by $\overline{\omega}(G)$, is the maximum cardinality among all twin-free cliques in $G$. Recall that
a \textit{$k$-edge coloring} of a graph $G$ is an assignment of $k$ colors $\{1,\ldots,k\}$ to the edges of $G$ such that no two adjacent edges have the same color, and
 the \textit{edge chromatic number} $\chi'(G)$ of a graph $G$ is the smallest integer $k$ such that $G$ has a $k$-edge coloring. A graph $G$ is called \textit{Class} 1, if $\chi'(G)=\Delta(G)$ and
it is of \textit{class} 2 if $\chi'(G)=\Delta(G)+1$ A graph $G$ is called \textit{overfull} if $|E(G)|>\lfloor\frac{|V(G)|}{2}\rfloor\Delta(G)$. For any undefined notation or terminology in graph theory, we refer the reader to \cite {west}.

\textit{The essential annihilating-ideal graph} of a ring $R$ is defined as the graph   $\mathcal{EG}(R)$ and two distinct vertices $I,J$ are joined by an edge if and only if $Ann(IJ)\leq_e R$. This graph was first introduced and studied by Nazim and Rehman in \cite{asl} as the ideal version of essential graph (see \cite{nikm}). The authors proved many interesting results for $\mathcal{EG}(R)$. For instance, they proved $\mathcal{EG}(R)$ is
always connected of diameter at most three and girth at most four (if it has a
cycle). Furthermore, they classified rings $R$ whose $\mathcal{EG}(R)$ is  star or complete.  They also investigated rings $R$ for which $\mathcal{EG}(R)$ is tree, unicycle, split, outerplanar and planar. This paper aims to discuss the coloring of $\mathcal{EG}(R)$. We show that $\mathcal{EG}(R)$ is weakly perfect, if $R$ is Noetherian and an explicit formula for the chromatic number of $\mathcal{EG}(R)$ is given.  Moreover, bipartite essential annihilating-ideal graph are characterized. Finally, twin-free clique number and edge chromatic number of $\mathcal{EG}(R)$ are computed.


{\begin{center}{\section{Main results}}\end{center}}\vspace{-2mm}

First of all we are going to compute the clique number of  $\mathcal{EG}(R)$, when $R$ is Noetherian. For this aim, we need the following lemmas.

\begin{lem}\label{lem1}
Let $I$ be an ideal of a ring $R$. If $I$ is nilpotent, then $I\subseteq {\rm Nil}(R)$. Moreover, the converse is also true if $I$ is finitely generated.
\end{lem}

\begin{proof}
{If $I$ is nilpotent, then it is clear that $I\subseteq {\rm Nil}(R)$. To prove the converse, let $I=R(x_1,\ldots, x_n)\subseteq {\rm Nil}(R)$, for some elements $x_1,\ldots, x_n\in R$. Then  there exists a positive integer $t_i$ such that $x_i^{t_i}=0$, for every $1\leq i \leq n$. Let $t\geq max\{t_1,\ldots, t_n\}$. Then obviously, $I^t=(0)$.
}
\end{proof}
 \begin{lem}\label{lem2}
Let $I$ be an ideal of a ring $R$. If $I$ is nilpotent, then $Ann(I)\leq_e R$. The converse is also true if $R$ is Noetherian.
\end{lem}
\begin{proof}
{Let $I$ be a nilpotent ideal of $R$. If $I=(0)$, then there is nothing to prove. Hence assume that $I\neq (0)$. By the proof of part 1 in \cite[Lemma 3.1]{asl}, $Ann(I)\leq_e R$. Conversely, suppose that $I$ is an ideal of a Noetherian ring $R$ and $Ann(I)\leq_e R$.  Since $R$ is Noetherian, there exist $x_1,\ldots, x_n\in R$ such that $I=R(x_1,\ldots, x_n)$. We show that $x_i$ is nilpotent, for every $1\leq i \leq n$. Let $\mathcal{J}=\{x\in R|\,\, Ann(x)\leq_e R\}$. It is not hard to see $\mathcal{J}$ is a proper ideal of $R$. We claim that $\mathcal{J}$ is nilpotent. It is enough to prove that every element of $\mathcal{J}$ is nilpotent. Let $\mathcal{B}=\{Ann(x)\leq_e R|\,\,x\in R\}$. For a given $x\in \mathcal{J}$, let $n\in \mathbb{N}$ be such that $Ann(x^n)$ is  maximal in $\mathcal{B}$. If we let $y=x^n$, then $y\in \mathcal{J}$ and $Ann(y)=Ann(y^2)$. Now, we have $Ry\cap Ann(y)=(0)$ which implies that $Ry=(0)$, i.e., $x^n=0$ and so the claim is proved. Since $Ann(I)\leq_e R$, we deduce that  $Ann(x_i)\leq_e R$ and hence $x_i\in \mathcal{J}$, for every $1\leq i \leq n$. Therefore, $I$ is nilpotent, as every $x_i$ is nilpotent.
}
\end{proof}


The next result states that for every Noetherian ring $R$ the graph $\mathcal{EG}(R)$ is weakly perfect.

\begin{thm}\label{colcol}
Let $R$ be a Noetherian ring and $\mathcal{A}=\{I \in A(R)^*|\,\, I\subseteq {\rm Nil}(R)\}$. Then the following statements hold:

$(1)$ If $|{\rm Min}(R)|\geq 2$, then $\omega(\mathcal{EG}(R))=\chi(\mathcal{EG}(R))=|{\rm Min}(R)|+|\mathcal{A}|.$

$(2)$ If $|{\rm Min}(R)|=1$, then $|\mathcal{A}|\leq\omega(\mathcal{EG}(R))=\chi(\mathcal{EG}(R))\leq |\mathcal{A}|+1.$

\end{thm}
\begin{proof}
{ If $R$ is reduced, then the result follows from \cite[Theorem 2.5]{asl} and \cite[Theorem 8]{aa}. Hence we may suppose that $\mathcal{A}\neq \varnothing$. By Lemma \ref{lem2} and part 1 in \cite[Lemma 3.1]{asl}, every vertex of $\mathcal{A}$ is adjacent to every other vertex of $\mathcal{EG}(R)$. Hence, if $|\mathcal{A}|=\infty$, then $\omega(\mathcal{EG}(R))=\chi\mathcal{EG}(R))=\infty$. So assume that $|\mathcal{A}|<\infty$ and $\mathcal{A}=\{J_1,\ldots,J_t\}$. We continue the proof in two following cases:

 $(1)$ $|{\rm Min}(R)|\geq 2$. Let ${\rm Min}(R)=\{\mathfrak{p}_1,\ldots,\mathfrak{p}_n\}$ (We note that $R$ is Noetherian and so $|{\rm Min}(R)|<\infty$) and $\mathfrak{\hat {p}}_i=\mathfrak{p}_1\ldots\mathfrak{p}_{i-1}\mathfrak{p}_{i+1}\ldots\mathfrak{p}_n$,  for every $1\leq i \leq n$. Since $\mathfrak{\hat{p}}_i\mathfrak{\hat {p}}_j\subseteq {\rm Nil}(R)$, we deduce from Lemmas \ref{lem1} and \ref{lem2} that $\mathfrak{\hat{p}}_i,\mathfrak{\hat {p}}_j$ are adjacent vertices in $\mathcal{EG}(R)$, for every $1\leq i\neq j \leq n$. Thus by part 1 of \cite[Lemma 3.1]{asl}, $\mathcal{EG}(R)[\mathcal{A}\cup\mathcal{B}]=K_{t+n}$, where $\mathcal{B}=\{\mathfrak{\hat{p}}_1,\ldots,\mathfrak{\hat{p}}_n\}$ and so $\omega(\mathcal{EG}(R))\geq t+n$. To complete the proof, we show that $\chi(\mathcal{EG}(R))\leq n+t$.  Define $c:V(\mathcal{EG}(R))\longrightarrow \{1,\ldots,n+t\}$ by $$c(I)=
  \begin{cases}
   \min\{1\leq i\leq n|\, I\nsubseteq \mathfrak{p}_i\} ;      & I\notin \mathcal{A},\,\,\, \\
   n+j;       & I=J_j,\, {\rm for\, some} \,1\leq j\leq t.\,\,\, \,\,
 \end{cases}$$

We show that $c$ is a proper vertex coloring for $\mathcal{EG}(R)$. Suppose to the contrary, $I,J$ are adjacent vertices and $c(I)=c(J)=i$. Thus $Ann(IJ)\leq_e R$. Since $R$ is Noetherian, Lemma \ref{lem2} implies that $IJ \subseteq {\rm Nil}(R)$ and hence $I\subseteq \mathfrak{p}_i$ or $J\subseteq \mathfrak{p}_i$, a contradiction. Therefore, $\chi(\mathcal{EG}(R))\leq n+t$, as desired.

$(2)$ $|{\rm Min}(R)|=1$.  Let ${\rm Min}(R)=\{\mathfrak{p}\}$ and $\mathcal{C}=\{I\in A(R)^*|\,\,I\nsubseteq \mathfrak{p}\}$. If $|\mathcal{C}|<\infty$, then $R$ is an Artinian local ring, as $|\mathcal{A}|<\infty$ and so by \cite[Lemma 3.2]{asl}, $\mathcal{EG}(R)=K_t$. Hence assume that $|\mathcal{C}|=\infty$. We show that $\mathcal{EG}(R)=K_t\vee \overline{K}_{\infty}$, where $\overline{K}_{\infty}$ is the complement of ${K}_{\infty}$. We have only to prove that $\mathcal{C}$ is an independent set. For if not, there exist $I,J\in \mathcal{C}$ such that $Ann(IJ)\leq_e R$. By Lemmas \ref{lem1} and \ref{lem2}, $IJ\subseteq \mathfrak{p}$ and thus $I\subseteq \mathfrak{p}$ or $J\subseteq \mathfrak{p}$ which contradicts $I,J\in \mathcal{C}$. Therefore, $\mathcal{EG}(R)=K_t\vee \overline{K}_{\infty}$ and so $\omega(\mathcal{EG}(R))=\chi(\mathcal{EG}(R))=t+1.$~~
}
\end{proof}

%


The next example investigates Theorem \ref{colcol} in use.

\begin{example}
{\rm $(1)$ Let $n$ be a positive integer with prime factorization  $n=p_1^{\alpha_1}\ldots p_k^{\alpha_k}$, where $\alpha_i$'s are positive integers, $p_i$'s are pairwise distinct primes and $k\geq 2$. If $R=\mathbb{Z}_n$, then by part $(1)$ of Theorem \ref{colcol}, $$\omega(\mathcal{EG}(R))=\chi(\mathcal{EG}(R))=\prod_{i=1}^k\alpha_i+k-1,$$

where $|{\rm Min}(R)|=k$ and $|\mathcal{A}|=\prod_{i=1}^k\alpha_i-1$.

$(2)$ Let $R= \mathbb{Z}_2[X,Y]/(XY,X^2)$, $x=X+(XY,X^2)$ and $y=Y+(XY,X^2)$. Then  $\mathrm{Min}(R)=\{\mathfrak{p}\}$, where $\mathfrak{p}=\{0,x\}$ and $Z(R)=R(x,y)$. By the notations in the proof of part $(2)$ in Theorem \ref{colcol}, $|\mathcal{A}|=1$ and $|\mathcal{C}|=\infty$. Therefore, $\mathcal{EG}(R)=K_{1}\vee\overline{K}_\infty$ and $\omega(\mathcal{EG}(R))=\chi(\mathcal{EG}(R))= 1+1.$
}
\end{example}

Next, we classify rings $R$ whose essential annihilating-ideal graphs have finite clique numbers.

\begin{thm}\label{colorin}
Let $R$ be a ring and $\omega(\mathcal{EG}(R))<\infty$. Then the following statements hold:

$(1)$ If $R$ is reduced, then $\omega(\mathcal{EG}(R))=\chi(\mathcal{EG}(R))=|{\rm Min}(R)|$.

$(2)$ If $R$ is non-reduced,  then the following statements are equivalent:

$(i)$ $\omega(\mathcal{EG}(R))=\chi(\mathcal{EG}(R))=|A(R)^*|$; Indeed, $\mathcal{EG}(R)=K_n$, for some positive integer $n$.

$(ii)$  Either $R=\mathbb{F}_1\times \mathbb{F}_2$, where $\mathbb{F}_1,\mathbb{F}_2$ are two fields or $Z(R)={\rm Nil}(R)$.
\end{thm}
\begin{proof}
{$(1)$ It follows from \cite[Theorem 8]{aa} and \cite[Theorem 2.5]{asl}.

$(2)$ $(i)\Rightarrow (ii)$ Suppose that $\mathcal{EG}(R)=K_n$,
for some positive integer $n$. By
\cite[Theorem 1.4]{beh1}, $R$ is
 Artinian. Now, the result follows from \cite[Theorem 3.6]{asl}.

$(ii)\Rightarrow (i)$ If  $R=\mathbb{F}_1\times \mathbb{F}_2$, where $\mathbb{F}_1,\mathbb{F}_2$ are two fields, then $\mathcal{EG}(R)=K_2$. Hence suppose that $\omega(\mathcal{EG}(R))<\infty$ and $Z(R)={\rm Nil}(R)$. We first claim that ${\rm Nil}(R)$ is finitely generated. For if not, there exists an infinite subset $\{a_i\}_{i \in \Lambda}$ of ${\rm Nil}(R)$ such that $Ra_i\neq Ra_j$, for every $i\neq j$. By \cite[Lemma 3.1]{asl}, $\{Ra_i\}_{i \in \Lambda}$ forms an infinite clique in $\mathcal{EG}(R)$, a contradiction. The claim now is proved. By \cite[Lemma 21.8]{sharp}, ${\rm Nil}(R)$ is nilpotent. Let $\mathcal{A}=\{I \in A(R)^*|\,\, I\subseteq {\rm Nil}(R)\}$. Since $\omega(\mathcal{EG}(R))<\infty$, we conclude that $\mathcal{A}$ is a finite set. Hence, if $x\in {\rm Nil}(R)^*$, then we may consider $Rx$ as an Artnian $R$-module. Moreover, $Z(R)={\rm Nil}(R)$ and $\omega(\mathcal{EG}(R))<\infty$ imply that $Ann(x)$ is an Artnian $R$-module, too. The isomorphism $Rx\simeq \frac{R}{Ann(x)}$ shows that $R$ is an Artinian ring. Since $Z(R)={\rm Nil}(R)$, $R$ is local. Finally, \cite[Lemma 3.2]{asl} completes the proof.
}
\end{proof}

In the sequel, we focus on rings  in which $\omega(\mathcal{EG}(R))=\chi(\mathcal{EG}(R))=2$.

\begin{remark}\label{rem}
{\rm In \cite[Theorem 3.7]{asl}, it has been claimed that ``for a ring $R$ with at least one minimal ideal we have $\mathcal{EG}(R)=K_{m,n}$, where $m,n\geq 2$ if and only if $R=D\times S$, where $D,S$ are two integral domains which are not fields". It should be noted that if $R=D\times S$, where $D,S$ are two integral domains which are not fields, then $R$ does not have any minimal ideal, in other words, ${\rm Soc}(R)=0$}.
\end{remark}

The next result fixes \cite[Theorem 3.7]{asl}.

\begin{lem}\label{soceee}
Let $R$ be a ring. Then $\mathcal{EG}(R)=K_{\infty,\infty}$ if and only if $R$ is reduced, $|{\rm Min}(R)|=2$ and ${\rm Soc}(R)=0$.
\end{lem}
\begin{proof}
{Let $\mathcal{EG}(R)=K_{\infty,\infty}$. We first show that $R$ is reduced. If $0\neq x \in {\rm Nil}(R)$, then $Rx$ is adjacent to every other vertex, a contradiction and so ${\rm Nil}(R)=(0)$. Next, we prove that $|{\rm Min}(R)|=2$. Since $R$ is  reduced, by  \cite[Corollary 2.4]{Huckaba}, $Z(R)=\cup_{\mathfrak{p}\in\mathrm{Min}(R) }\mathfrak{p}$. Suppose that $\mathfrak{p}_1$, $\mathfrak{p}_2$ and $\mathfrak{p}_3$ are three distinct minimal prime ideals. If $x\in \mathfrak{p}_1\setminus \mathfrak{p}_2\cup\mathfrak{p}_3$, then $\mathrm{Ann}(x)\subset \mathfrak{p}_2\cap\mathfrak{p}_3$. Let $0\neq y\in \mathrm{Ann}(x)$. Since $Rx\mathfrak{p}_2\neq (0)$,  let $a\in Rx\cap\mathfrak{p}_2$. As $R$ is  reduced, $Rx\cap\mathrm{Ann}(x)=(0)$. This implies that $a\notin \mathrm{Ann}(x)$. Since $a, y\in\mathfrak{p}_2 $, we have $a+y=z\in \mathfrak{p}_2 $ and so $\mathrm{Ann}(z)\neq (0)$. By  \cite[Corollary 2.2]{Huckaba}, $hz=0$ for some $h\notin \mathfrak{p}_2$. Since $\mathrm{Ann}(z)=\mathrm{Ann}(y)\cap\mathrm{Ann}(a)$, $Ra-Ry-Rh-Ra$ is a cycle of length 3 which is impossible. Thus $|{\rm Min}(R)|\leq 2$. It is clear that $|{\rm Min}(R)|=1$ means $R$ is an integral domain and hence $|{\rm Min}(R)|=2$. To see ${\rm Soc}(R)=0$, we consider two following cases:

$(1)$ $R$ is decomposable. Thus $R=R_1\times R_2$, where $R_1,R_2$ are two rings. If $R_1$ is not an integral domain, then there exists an ideal $I\in A(R_1)^*$. If $I^2=0$, then $(I,(0))$ is adjacent to all other vertices which contradicts $\mathcal{EG}(R)=K_{\infty,\infty}$. Thus $I\neq Ann(I)$. Thus $(I,(0))-((0),R_2)-(Ann(I),(0))-(I,(0))$ forms a triangle in $\mathcal{EG}(R)$ which is impossible and so $R_1$ is an integral domain. Similarly,  $R_2$ is an integral domain, too. Hence $R=D_1\times D_2$, where $D_1,D_2$ are two integral domains. By Remark \ref{rem}, ${\rm Soc}(R)=0$.

$(2)$ $R$ is not decomposable. If $I$ is a non-zero minimal ideal of $R$, then either $I^2=(0)$ or $I^2=I$. Since $R$ is reduced, $I^2=(0)$ implies that $I=(0)$, a contradiction. If  $I^2=I$, then by Brauer's Lemma (see \cite[10.22]{lam}), $R=Re\times R(1-e) $, where $e$ is an idempotent element of $R$ which is impossible. Therefore, $R$ does not have any minimal ideal and so ${\rm Soc}(R)=0$.

Conversely,  suppose that $R$ is reduced, ${\rm Min}(R)=\{\mathfrak{p}_1,\mathfrak{p}_2\}$ and ${\rm Soc}(R)=0$. Since $R$ is reduced, we have  $Z(R)=\mathfrak{p}_1\cup\mathfrak{p}_2$ and $\mathfrak{p}_1\cap\mathfrak{p}_2=(0)$, by \cite[Corollary 2.4]{Huckaba}. Let $V_1, V_2$ be the sets of all non-zero annihilating-ideals contained in $\mathfrak{p}_1, \mathfrak{p}_2$, respectively. It is easy to check that $\mathcal{EG}(R)=K_{|V_1|,|V_2|}$. Finally, we have $|V_1|=|V_2|=\infty$, as ${\rm Soc}(R)=0$.
}
\end{proof}
\begin{cor}\label{soc}
Let $R$ be a ring which contains a non-trivial idempotent $e$. Then $\mathcal{EG}(R)=K_{\infty,\infty}$ if and only if $R=D_1\times D_2$, where $D_1,D_2$ are two integral domains which are not fields.
\end{cor}
\begin{proof}
{One side is clear. Conversely,  it is not hard to see that $R=Re\times R(1-e) $, where $e$ is a non-trivial idempotent element of $R$. A similar argument to that of Case $1$ in Lemma \ref{soceee} implies that both of $Re$ and $R(1-e)$ are integral domains. By Lemma \ref{soceee}, ${\rm Soc}(R)=0$ and so non of $Re$ and $R(1-e)$ are fields.
}
\end{proof}
\begin{lem}\label{dobakh}
Let $R$ be a ring. Then $\mathcal{EG}(R)$ is bipartite if and only if it is complete bipartite.
\end{lem}
\begin{proof}
{Suppose that $\mathcal{EG}(R)$ is bipartite with parts $V_1,V_2$. If there exists  $0\neq x \in {\rm Nil}(R)$, then $Rx$ is a vertex of $\mathcal{EG}(R)$. With out loss of generality, assume that $Rx\in V_1$. Hence $|V_1|=1$ and $Rx$ is adjacent to all vertices contained in $V_2$ and so $\mathcal{EG}(R)$ is star. Thus, we may suppose that ${\rm Nil}(R)=(0)$. By a similar argument to that in the proof of Lemma \ref{soceee}, $|{\rm Min}(R)|=2$ and $\mathcal{EG}(R)$ is  complete bipartite. The converse is trivial.
}
\end{proof}
\begin{lem}\label{triangle}
Let $R$ be a ring. Then $\omega(\mathcal{EG}(R))=2$ if and only if $\chi(\mathcal{EG}(R))=2$.
\end{lem}
\begin{proof}
{Suppose that $\omega(\mathcal{EG}(R))=2$. If ${\rm Nil}(R)=(0)$, then by a similar argument to that in the proof of Lemma \ref{soceee},  $\mathcal{EG}(R)$ is  complete bipartite and so $\chi(\mathcal{EG}(R))=2$. Hence let ${\rm Nil}(R)\neq (0)$. Then ${\rm Nil}(R)$ should be a principal ideal. If ${\rm Nil}(R)$ is a minimal ideal, then it is adjacent to every other vertex and since $\omega(\mathcal{EG}(R))=2$, we infer that $\mathcal{EG}(R)$ is star. If ${\rm Nil}(R)$ is  not minimal, then it contains exactly one ideal ${\rm Nil}(R)^2$. Now, $\omega(\mathcal{EG}(R))=2$ implies that $A(R)^*=\{{\rm Nil}(R),{\rm Nil}(R)^2\}$ and so $\mathcal{EG}(R)=K_2$. It is evident $\chi(\mathcal{EG}(R))=2$, if  either $\mathcal{EG}(R)$ is star or $K_2$. The converse is obvious.
}
\end{proof}
\begin{lem}\label{mvan}
Let $R$ be a ring. If $\mathcal{EG}(R)=K_{m,n}$, then $m,n\in \{1,\infty\}$.
\end{lem}
\begin{proof}
{If $R$ is reduced, then $|{\rm Min}(R)|=2$ and $\mathcal{EG}(R)$ is  complete bipartite. Moreover, $R$ has at most two minimal ideals. If  ${\rm Soc}(R)=0$, then $\mathcal{EG}(R)=K_{\infty,\infty}$, by Lemma \ref{soceee}. If $R$ has exactly one minimal ideal, then it is not hard to see that $R=\mathbb{F}\times D$, where $\mathbb{F}$ is a field and $D$ is an integral domain. In this case, $\mathcal{EG}(R)=K_{1,\infty}$. If $R$ has two minimal ideals, then $R=\mathbb{F}_1\times \mathbb{F}_2$, where $\mathbb{F}_1,\mathbb{F}_2$ are two fields and so $\mathcal{EG}(R)=K_{1,1}$. If $R$ is not reduced, then  ${\rm Nil}(R)$ is a principal ideal. If ${\rm Nil}(R)$ is a minimal ideal, then it is adjacent to every other vertex and since $\omega(\mathcal{EG}(R))=2$, $\mathcal{EG}(R)$ is star. By \cite[Theorem 3.5]{asl}, $\mathcal{EG}(R)=K_{1,\infty}$, in this case. If ${\rm Nil}(R)$ is  not minimal, then it contains exactly one ideal ${\rm Nil}(R)^2$. Therefore, $A(R)^*=\{{\rm Nil}(R),{\rm Nil}(R)^2\}$ and thus $\mathcal{EG}(R)=K_{1,1}$.
}
\end{proof}

Now, we are in a position to completely characterize rings $R$ whose essential annihilating-ideal graphs have chromatic number $2$.

\begin{thm}\label{fital}
Let $R$ be a ring with at least two non-trivial annihilating ideals. Then the following statements are equivalent:

$(1)$ $\omega(\mathcal{EG}(R))=2$.

$(2)$ $\chi(\mathcal{EG}(R))=2$.

$(3)$ $\mathcal{EG}(R)$ is bipartite.

$(4)$ $\mathcal{EG}(R)$ is complete bipartite.

$(5)$ $\mathcal{EG}(R)\in \{K_{1,1}, K_{1,\infty}, K_{\infty,\infty}\}$.

$(6)$ $R$ is one of the following rings:

$(i)$ $R=\mathbb{F}_1\times \mathbb{F}_2$, where $\mathbb{F}_1,\mathbb{F}_2$ are two fields.

$(ii)$ $R=\mathbb{F}\times D$, where $\mathbb{F}$ is a field and $D$ is an integral domain.

$(iii)$ $R$ is reduced, $|{\rm Min}(R)|=2$ and ${\rm Soc}(R)=0$.

$(iv)$ $R$ is a ring with exactly two non-trivial ideals ${\rm Nil}(R),{\rm Nil}(R)^2$.

$(v)$ $R$ has a minimal $I_1$ such that $I_1$ is not an essential ideal of $R$, $I_1^2=(0)$ and for any non-zero annihilating-ideal $I_2$ of $R$, $Ann(I_2)=I_1$.
\end{thm}
\begin{proof}
{The result follows from Lemmas \ref{soceee}, \ref{dobakh}, \ref{triangle}, \ref{mvan} and \cite[Theorem 3.5]{asl}.
}
\end{proof}

The next result computes twin-free clique number of $\mathcal{EG}(R)$, when $R$ is Artinian.

\begin{thm}\label{twin}
Let $R$ be an Artinian ring. Then the following statements hold:

$(1)$ If $R$ is reduced, then

$(i)$ $\overline{\omega}(\mathcal{EG}(R))=1$, if $|{\rm Max}(R)|=2$.

$(ii)$ $\overline{\omega}(\mathcal{EG}(R))=|{\rm Max}(R)|$, if $|{\rm Max}(R)|>2$.

$(2)$ If $R$ is non-reduced, then

$(i)$ $\overline{\omega}(\mathcal{EG}(R))=1$, if $|{\rm Max}(R)|=1$.

$(ii)$ $\overline{\omega}(\mathcal{EG}(R))=|{\rm Max}(R)|+1$.
\end{thm}
\begin{proof}
{If $R$ is a non-reduced local ring or reduced with exactly two maximal ideals, then the result follows from \cite[Lemma 3.2]{asl}, as $\mathcal{EG}(R))$ is complete. Hence, we may suppose that there exists a positive integer $n\geq 2$ such that $R=R_1\times\cdots\times R_n$ and every $R_i$ is an Artinian local ring. Consider the following partition for $V(\mathcal{EG}(R))$:

$\mathcal{A}=\{I\in A(R)^*|\,\, I\subseteq {\rm Nil}(R)\}$,

$\mathcal{B}=\{J_1,\ldots,J_n\}$, where $J_i=(0)\times\cdots\times(0)\times R_i\times (0) \times\cdots\times (0)$,

$\mathcal{B}_1=\{(R_1,I_2,\ldots,I_n)\in A(R)^*|\,\, I_i\leq R_i, 2\leq i \leq n\}\setminus J_1$,

$\mathcal{B}_2=\{(I_1,R_2,I_3,\ldots,I_n)\in A(R)^*|\,\, I_1\neq R_1, I_i\leq R_i, 3\leq i \leq n\}\setminus J_2$,

.

.

.

$\mathcal{B}_n=\{(I_1,I_2,\ldots,I_{n-1}, R_n)\in A(R)^*|\,\, I_i\neq R_i, 1\leq i \leq n-1\}\setminus J_n$.

We first note that $\mathcal{EG}(R)[\mathcal{B}]$ is a complete graph, as $J_iJ_j=(0)$, for all $1\leq i\neq j \leq n$. Moreover, every $\mathcal{B}_i$ is an independent set, for $1\leq i \leq n$. To see this, and with no loss of generality, let $I,J \in \mathcal{B}_1$. Thus $IJ\in \mathcal{B}_1\cup \{J_1\}$ and so $IJ$ is not nilpotent. Hence, by Lemma \ref{lem2}, $I,J$ are not adjacent and $\mathcal{EG}(R)[\mathcal{B}_i]$ is totally disconnected, for $1\leq i \leq n$. Next, we show that $\mathcal{B}$ contains no true twins. Let $J_i,J_j$ be two arbitrary vertices of $\mathcal{B}$. With no loss of generality, suppose that $i<j$. Then $(R_1,\ldots,R_{i-1},(0),R_{i+1},\ldots,R_j,\ldots,R_n)\in N(J_i)\setminus N(J_j)$, as desired. We continue the proof in two following cases:

$(1)$ $R$ is reduced and $|{\rm Max}(R)|>2$. Then $\mathcal{A}=\varnothing$ and thus, $\overline{\omega}(\mathcal{EG}(R))=|\mathcal{B}|=|{\rm Max}(R)|$.

$(2)$ $R$ is non-reduced and $|{\rm Max}(R)|\geq 2$. By Lemma \ref{lem1} and \cite[Lemma 3.1]{asl}, $\mathcal{EG}(R)[\mathcal{A}]$ is a complete graph and every vertex in $\mathcal{A}$ is adjacent to every other vertex.  Hence, $\mathcal{B}\cup \{I\}$ is a twin free clique of maximum cardinality, for every $I\in \mathcal{A}$. Therefore, $\overline{\omega}(\mathcal{EG}(R))=|{\rm Max}(R)|+1$.
}
\end{proof}

The rest of this paper is devoted to the edge chromatic number of $\mathcal{EG}(R)$.

\begin{remark}
{\rm By Vizing's Theorem (see \cite[p. 16]{yap}),  simple undirect graphs are partitioned into two classes Class 1 and Class 2. If a graph $G$ contains a vertex of infinite degree, then clearly, it is  Class 1. Hence to determine whether essential annihilating-ideal graphs are Class $1$ or Class $2$, we may suppose that every vertex of $\mathcal{EG}(R)$ has a finite degree. Then, by \cite[Theorem 1.4]{beh1}, $R$ has finitely many ideals. Such rings are obviously Artinian. Since Artinian local rings have complete essential annihilating-ideal graphs, both of  Class $1$  and Class $2$  appear among $\mathcal{EG}(R)$'s. Indeed, if $R$ is an Artinian local ring containing an even number of non-trivial ideals, then $\mathcal{EG}(R)$ is Class $1$, but if it contains an odd number of non-trivial ideals, then $\mathcal{EG}(R)$ is Class $2$. Hence suppose that $R$ is  non-local.  Then
there exists a positive integer $n\geq 2$ such that $R\cong
R_1\times\cdots\times R_n$, where every $R_i$ is an Artinian
local ring. If every $R_i$ has $t_i$ proper ideals, then $V(\mathcal{EG}(R))=\prod_{i=1}^n(t_i+1)-2$. If $R$ is non-reduced, then  ${\rm Nil}(R)$ is adjacent to every other vertex. If there exists $1\leq i \leq n$ such that $t_i$ is odd, then $\mathcal{EG}(R)$ is Class $1$, see \cite[p. 45]{plant}. Thus two cases remain:

$(i)$ $R$ is reduced.

$(ii)$ $R$ is a non-reduced ring and $|V(\mathcal{EG}(R))|$ is odd.}
\end{remark}

The following lemma is a criteria for a graph to be Class 1.

\begin{lem}\label{firstkindgraph} {\rm(\cite[Corollary 5.4]{beineke})}
Let $G$ be a simple graph. Suppose that for every
vertex $u$ of maximum degree, there exists an edge $\{u,v\}$ such
that $\Delta(G)-deg(v)+2$ is more than the number of vertices with
maximum degree in $G$. Then $\chi'(G)=\Delta(G)$.
\end{lem}

Since every (non-local) Artinian reduced ring is a direct product of finitely many fields, we state the following result.

\begin{thm}\label{edgechromaticnumber}
Let $R= \mathbb{F}_1\times\cdots\times \mathbb{F}_n$, where every $\mathbb{F}_i$ is a
field and $n\geq 2$. Then $\mathcal{EG}(R)$ is Class $1$.
\end{thm}
\begin{proof}
{If $n=2$, then there is nothing to prove, as $\mathcal{EG}(R)=K_2$. Thus, we may assume that $n\geq 3$. It is not hard to see that every element of $\mathcal{A}=\{(0)\times\cdots\times (0)\times
F_i\times (0)\times\cdots\times (0)\,|\,1\leq i\leq n\}$ is a vertex of maximum degree and $|\mathcal{A}|=n$. Also,
 $\Delta(\mathcal{EG}(R))=2^{n-1}-1$.
Let $u$ be a vertex of maximum degree in
$\mathcal{EG}(R)$. With out loss of generality, suppose
that $u=\mathbb{F}_1\times(0)\times\cdots\times (0)$. Then
$\Delta(\mathcal{EG}(R))-d(v)+2=2^{n-1}>n$, for  the vertex $v=(0)\times \mathbb{F}_2\times\cdots\times \mathbb{F}_n$.  Lemma
\ref{firstkindgraph} now implies that $\mathcal{EG}(R)$ is Class $1$.
}
\end{proof}

We proceed with non-reduced ring case.

\begin{thm}\label{clas1}
Let $R= R_1\times\cdots\times R_n$, where every $R_i$ is an
Artinian local ring, $n\geq 2$ and every $R_i$ has exactly $t\geq 2$ proper ideals. If $t$ is a fixed even number, then $\mathcal{EG}(R)$ is Class $1$, for enough large $n$.
\end{thm}
\begin{proof}
{Let $\mathcal{A}=\{I \in A(R)^*|\,\, I\subseteq {\rm Nil}(R)\}$. By Lemmas \ref{lem1} and \ref{lem2}, every vertex $J_1\times\cdots\times J_n \in \mathcal{A}$ is adjacent to all  vertices of $\mathcal{EG}(R)$ and  $J_i\neq
R_i$, for every $i$. By an easy calculation,
$\Delta(\mathcal{EG}(R))=(t+1)^n-3$ and the number of vertices with maximum
degree is $t^n-1$. For every vertex $u$ of
$\mathcal{EG}(R)$ of maximum degree, choose
$v=R_1\times\cdots\times R_{n-1}\times (0)$. Then
$deg(v)=(t+1)(t^{n-1})-1=t^n+t^{n-1}-1$.
Thus
$$\Delta(\mathcal{EG}(R))-deg(v)+2=(t+1)^n-t^n-t^{n-1}.$$
Now, if $$(t+1)^n-2t^n-t^{n-1}+1>0, ~~~{\rm (I})$$  then by Lemma \ref{firstkindgraph}, $\mathcal{EG}(R)$ is Class $1$.
Since $t$ is constant, (I) is always hold, for enough large $n$.
}
\end{proof}

The following table gives suitable $n$'s for each $t$.

\begin{center}
\begin{tabular}{|c|c|c}
\hline
$t$ & suitable $n$ \\
\hline
$t=2$ & $n\geq 3$\\
\hline
$t=4$ & $n\geq 4$ \\
\hline
$t=6$ & $n\geq 6$ \\
\hline
$t=8$ & $n\geq 7$ \\
\hline
. & . \\
\hline
. & . \\
\hline
. & . \\
\hline
$t=44$ & $n\geq 32$ \\
\hline
. & . \\
\hline
\end{tabular}
\end{center}

Theorem \ref{clas1} fails if $n$ is not large enough, see the next example.

\begin{example}
Let $R=\mathbb{Z}_{9}\times\mathbb{Z}_{25}$. Then $t=2$ and $n=2$. We have that $|E(\mathcal{EG}(R))|=19$, $|V(\mathcal{EG}(R))|=7$ and $|\Delta(\mathcal{EG}(R))|=6$. Then $\mathcal{EG}(R)$ is overfull and so it is Class $2$.
\end{example}

%

We close this paper with the following result which recognizes Class $2$ essential annihilating-ideal graphs, when Artinian ring $R$ has exactly two maximal ideals and $|V(\mathcal{EG}(R))|$ is odd. However, the  general case-$|{\rm Max}(R)|=n\geq 3$ and odd $|V(\mathcal{EG}(R))|$-becomes very complicated and it remains still open.
\begin{thm}
Suppose that $R=R_1\times R_2$, where $R_i$ is an Artinian local ring with $t_i\geq 2$ proper ideals and $t_i$ is even for $i=1, 2$. Then $\mathcal{EG}(R)$ is Class $2$ if and only if $t_1=t_2=2$.
\end{thm}
\begin{proof}{
Since ${\rm Nil}(R)$ is adjacent to every other vertex, by \cite[Theorem p. 46]{plant} it is enough to show that ``$\mathcal{EG}(R)$ is overfull if and only if $t_1=t_2=2$". Let

$\mathcal{A}=\{I\in A(R)^*|\,\, I\subseteq {\rm Nil}(R)\}$,

$\mathcal{B}_1=\{(R_1,I_2)\in A(R)^*|\,\, I_2\neq R_2\}$ and

$\mathcal{B}_2=\{(I_1,R_2)\in A(R)^*|\,\, I_1\neq R_1\}$.

 One may show that $|\mathcal{A}|=t_1t_2-1$ and $deg(u)=\Delta(\mathcal{EG}(R))=t_1t_2+t_1+t_2-2$, for each vertex $u\in \mathcal{A}$. Moreover, $|\mathcal{B}_1|=t_2$, $deg(u)=t_1t_2+t_1-1$, for each vertex $u\in \mathcal{B}_1$ and $|\mathcal{B}_2|=t_1$, $deg(u)=t_1t_2+t_2-1$, for each vertex $u\in \mathcal{B}_2$. Thus $\mathcal{EG}(R)$ is overfull if and only if $$|E(\mathcal{EG}(R))|>\lfloor\frac{|V(\mathcal{EG}(R))|}{2}\rfloor\Delta(\mathcal{EG}(R))$$ if and only if $$\frac{(t_1t_2-1)(t_1t_2+t_1+t_2-2)+t_2(t_1t_2+t_1-1)+t_1(t_1t_2+t_2-1)}{2}>$$
 $$\lfloor\frac{t_1t_2+t_1+t_2-1}{2}\rfloor(t_1t_2+t_1+t_2-2)$$ if and only if $$t_1=t_2=2.$$
}
\end{proof}



{}

\end{document}